\documentclass[12pt,letterpaper]{amsart}

\usepackage[utf8]{inputenc}

\usepackage[american]{babel}

\usepackage{hyperref}
\usepackage{anysize}
\marginsize{2cm}{2cm}{2cm}{2cm}

\usepackage{amsmath}
\usepackage{amssymb}
\usepackage{latexsym}
\usepackage{mathrsfs}

\usepackage{amsthm}

\usepackage{color}

\usepackage{indentfirst} 

\usepackage{ifpdf}

\ifpdf
\usepackage[pdftex]{graphicx}
\else
\usepackage{graphicx}
\fi

\theoremstyle{plain}
\newtheorem{theorem}{Theorem}[section]
\newtheorem{lemma}[theorem]{Lemma}
\newtheorem{corollary}[theorem]{Corollary}

\theoremstyle{definition}
\newtheorem{definition}[theorem]{Definition}
\newtheorem{conjecture}[theorem]{Conjecture}
\newtheorem{problem}[theorem]{Problem}

\theoremstyle{remark}
\newtheorem{remark}[theorem]{Remark}

\numberwithin{equation}{section}

\newcounter{fig}
\newcommand{\f}{\refstepcounter{fig} Fig. \arabic{fig}. }

\title{Kadets type theorems for partitions of a convex body}

\author{Arseniy~Akopyan \and Roman Karasev}
\thanks{The research of A.V.~Akopyan is supported by the President's of Russian Federation grant MK-113.2010.1, the Russian Foundation for Basic Research grants 10-01-00096 and 11-01-00735, and the Russian government project 11.G34.31.0053.\\
The research of R.N.~Karasev is supported by the Dynasty Foundation, the President's of Russian Federation grant MK-113.2010.1, the Russian Foundation for Basic Research grants 10-01-00096 and 10-01-00139, the Federal Program ``Scientific and scientific-pedagogical staff of innovative Russia'' 2009--2013, and the Russian government project 11.G34.31.0053.}

\address{Arseniy Akopyan, Institute for Information Transmission Problems RAS\\ 
			Bolshoy Karetny per. 19, Moscow, Russia 127994  \\ \newline
Laboratory of Discrete and Computational Geometry, Yaroslavl' State University, Sovetskaya st. 14, Yaroslavl', Russia 150000}
\email{akopjan@gmail.com}

\address{Roman Karasev, Dept. of Mathematics, Moscow Institute of Physics and Technology, Institutskiy per. 9, Dolgoprudny, Russia 141700
 \newline
 Laboratory of Discrete and Computational Geometry, Yaroslavl' State University, Sovetskaya st. 14, Yaroslavl', Russia 150000}
\email{r\_n\_karasev@mail.ru}
\urladdr{http://www.rkarasev.ru/en/}
\begin{document}

\maketitle

\begin{abstract}
For convex partitions of a convex body $B$ we try to put a homothetic copy of $B$ into each set of the partition so that the sum of homothety coefficients is $\ge 1$. In the plane this can be done for arbitrary partition, while in higher dimensions we need certain restrictions on the partition.
\end{abstract}

\section{Introduction}

Alfred~Tarski~\cite{tarski1932further} proved that for any covering of the unit disk by planks (the sets $a\le \lambda(x)\le b$ for a linear function $\lambda$ and two reals $a<b$) the sum of plank widths is at least $2$. Th{\o}ger Bang in~\cite{bang1951asolution} generalized this result for covering of a convex body $B$ in $\mathbb R^d$ by planks showing that the sum of the widths is at least the width of $B$. He also posed the following question: Can the plank widths in the Euclidean metric be replaced by the widths relative to $B$ (as in Definition~\ref{inrad-def} below)?

Keith~Ball proved the conjecture of Bang in~\cite{ball1991thepalnk} for centrally symmetric bodies $B$ or, in other words, for arbitrary normed spaces and coverings of the unit ball. For possibly non-symmetric $B$, it is known (see~\cite{alexander1968aproblem}) that the Bang conjecture is equivalent to the Davenport conjecture: If a convex body $B$ is sliced by $n-1$ hyperplane cuts then there exists a piece that contains a translate of~$\frac{1}{n}B$.

In~\cite{bezdek1995solution,bezdek1996conway} Andr\'as Bezdek and K\'aroly Bezdek proved an analogue of the Davenport conjecture for binary partitions by hyperplanes. The difference is that they do not cut everything with every hyperplane; instead they divide one part into two parts and then proceed recursively.

One of the strongest results about coverings of a unit ball for the Hilbert (and finite dimensional Euclidean) space was proved by Vladimir Kadets in~\cite{kadets2005coverings} (see also~\cite{bezdek2007onageneralization} for the proof in the two-dimensional case using the idea from~\cite{tarski1932further}): For any convex covering $C_1,\ldots, C_k$ of the unit ball the sum of inscribed ball radii $\sum_{i=1}^k r(C_i)$ is at least $1$. 

The reader is referred to~\cite{bezdek2009tarski} for a detailed historical survey on the Tarski plank problem.

In this paper we prove analogues of the Kadets theorem for inscribing homothetic copies of a (not necessarily symmetric) convex body, replacing arbitrary coverings by certain convex partitions. By a \emph{partition} of a convex set $B$ we mean a covering of $B$ by a family of closed convex sets with disjoint interiors. In the two-dimensional case the analogue of the Kadets theorem for possibly non-symmetric bodies (Theorem~\ref{kadets-2d}) holds for any partition, while in higher dimensions we need additional restrictions on the partition. In other words, we are solving positively certain particular cases of~\cite[Problem~7.2]{bezdek2009tarski} about extending the Kadets theorem to Banach spaces. 

We work in finite-dimensional spaces. If one needs analogues for infinite-dimensional Banach spaces then the standard approximation argument works as in~\cite{kadets2005coverings}.

\section{Inductive partitions}
\label{ind-sec}

Let us describe the class of partitions for which an analogue of the Kadets theorem is true:

\begin{definition}
Call a convex partition $V_1\cup\dots\cup V_k$ of $\mathbb R^d$ \emph{inductive} if for any $1\le i\le k$ there exists an inductive partition $W_1\cup\dots\cup W_{i-1}\cup W_{i+1}\cup\dots\cup W_k$ such that $W_j\supseteq V_j$ for any $j\neq i$. A partition into one part $V_1=\mathbb R^d$ is assumed to be inductive.
\end{definition}

Now we define the \emph{inradius} relative to $B$:

\begin{definition}
\label{inrad-def}
Let $B\subset \mathbb R^d$ be a convex body. For a convex set $C\subseteq \mathbb R^d$ define the analogue of the inscribed ball radius as follows:
\[
r_B(C) = \sup \{h \ge 0: \exists t\in \mathbb R^d\ \text{such that}\ hB+t\subseteq C\},
\]
and put $r_B(C)=-\infty$ for empty $C$.
\end{definition}

Now we are ready to state one of the main results:

\begin{theorem}
\label{kadets-ind}
Let $B\subset \mathbb R^d$ be a convex body and let $C_1\cup\dots\cup C_k = B$ be induced from an inductive partition $V_1\cup\dots\cup V_k=\mathbb R^d$ (that is $C_i=V_i\cap B$ for any $i$). Then
\[
\sum_{i=1}^k r_B(C_i) \ge 1.
\]
\end{theorem}

Before proving this theorem we need a lemma about the inradius:

\begin{lemma}
\label{conv-inr}
Let a convex polytope $C\subset \mathbb R^d$ be defined by linear inequalities for $i=1,\ldots, m$:
\[
\lambda_i(x) \le 0.
\]
Denote by $C(\bar y)$ the polytope defined by the inequalities
\[
\lambda_i(x) + y_i \le 0,
\]
where $\bar y = (y_1, \ldots, y_m)$ is a vector of reals. Then $r_B(C(\bar y))$ is a concave function of $\bar y$.
\end{lemma}

\begin{proof}
Denote the set of indices $[m]=\{1,\ldots, m\}$. By the Helly theorem we have 
\[
r_B(C(\bar y)) = \inf_{I\subseteq [m],\ |I|\le d+1} r_B(C_I(\bar y)),
\]
where $C_I(\bar y)$ is defined by the inequalities $\lambda_i(x) + y_i \le 0$ for $i\in I$. 
The sets $C_I(\bar y)$ are either Cartesian products of a linear subspace $L\subset\mathbb R^d$ of positive dimension with a lower-dimensional polyhedral set $C'_I(\bar y)$, or simplicial cones, or simplices. In the first case we use induction on the dimension. In the second case we note that $r_B(C_I(\bar y))=+\infty$. In the third case the function $r_B(C_I(\bar y))$ is obviously linear. Hence for any $C_I(\bar y)$ the inradius $r_B(C_I(\bar y))$ is a concave function of $\bar y$. Therefore the inradius $r_B(C(\bar y))$ is concave as an infimum of concave functions.
\end{proof}

\begin{lemma}
\label{conv-inr-inters}
Let $C_1,\ldots, C_m$ be a family of convex bodies in $\mathbb R^d$. Then the inradius of the intersection of translates
\[
r_B\left ((C_1+y_1)\cap (C_2+y_2)\cap\dots\cap (C_m+y_m) \right)
\]
is a concave function of $\bar y = (y_1, \ldots, y_m)\in (\mathbb R^d)^{\times m}$.
\end{lemma}

\begin{proof}
When $C_i$'s are polytopes this is a particular case of Lemma~\ref{conv-inr}. The general case is made by approximating $C_i$'s by polytopes and going to the limit.
\end{proof}

\begin{remark}
In the above lemmas we actually prove that the set of vectors $\bar y$ such that the considered function of $\bar y$ is $>-\infty$ is a convex closed set.
\end{remark}

\begin{proof}[Proof of Theorem~\ref{kadets-ind}]
Let us vary the vector $y\in\mathbb R^d$ and define $C_i(y) = B\cap (V_i+y)$. The function $r(y) = \sum_{i=1}^k r_B(C_i(y))$ is a concave function of $y$ by Lemma~\ref{conv-inr-inters} and the set $Y = \{y : r(y)> -\infty \}$ is a convex closed set. If $y$ is in the boundary of $Y$ then at least one of $C_i(y)$ has empty interior. In this case we can omit the corresponding $V_i$ and consider a smaller partition $\{W_j\}_{j\neq i}$, which induces the same partition $\{(W_j+y)\cap B\}_{j\neq i}$ as $\{(V_j+y)\cap B\}$ up to sets with empty interior.

Thus by induction we have $r(y)\ge 1$ on $\partial Y$. Along with the concavity of $r(y)$ this implies $r(y)\ge 1$ on the whole $Y$ unless $Y$ is a halfspace. From the obvious formula (the sum is the Minkowski sum)
\[
Y=\bigcap_{i=1}^k (B + (-V_i))
\]
it follows that $Y$ can be a halfspace if and only if every $V_i$ contains the same halfspace. This is impossible unless $k=1$; but for $k=1$ the theorem is obviously true.
\end{proof}

\section{Affine partitions}
\label{aff-sec}

In this section we describe constructively a certain class of inductive partitions.

\begin{definition}
For a sequence of affine (linear with possible constant term) functions $F=\{\lambda_1,\ldots, \lambda_k\}$ define an \emph{affine partition} $P(F)$ of $\mathbb R^d$ by
\[
C_i = \{x\in\mathbb R^d : \forall j\neq i\ \text{we have}\ \lambda_i(x)\le \lambda_j(x) \}.
\]
An affine partition of a subset $X\subset \mathbb R^d$ is defined as a restriction of an affine partition of the whole $\mathbb R^d$.
\end{definition}

\begin{remark}
Affine partitions are also known as \emph{generalized Voronoi partitions} but we use the term \emph{affine} for brevity.
\end{remark}

\begin{corollary}
\label{kadets-aff}
Let $B\subset \mathbb R^d$ be a convex body and let $C_1\cup\dots\cup C_k = B$ be its affine partition. Then
\[
\sum_{i=1}^k r_B(C_i) \ge 1.
\]
\end{corollary}

\begin{proof}
It suffices to show that any affine partition is inductive. Starting from $V_1\cup\dots\cup V_k= \mathbb R^d$ defined by $\{\lambda_1, \ldots, \lambda_k\}$ we omit $\lambda_i$ from the list and obtain another affine partition $\{W_j\}_{j\neq i}$ such that $W_j\supseteq V_j$ for any $j\neq i$. So the induction step is possible.
\end{proof}

A straightforward generalization of an affine partition is a \emph{hierarchical affine partition}:

\begin{definition}
By induction: If in a hierarchical affine partition $C_1\cup\dots\cup C_k$ we partition some $C_i$ by an affine partition, we obtain again a hierarchical affine partition.
\end{definition}

Let us show that a hierarchical affine partition is a limit of affine partitions:

\begin{lemma}
Suppose $B$ is a convex body and $C_1\cup\dots\cup C_k=B$ is its hierarchical affine partitions. Then this partition can be approximated by an affine partition with arbitrary precision in the Hausdorff metric. 
\end{lemma}

\begin{proof}
From the definition we know that there exists a graded tree $T$ with an affine function $\lambda_v$ in every vertex $v\in T$ such that the sets $C_i$ correspond to the leaves $\ell_i$ of $T$; the condition $x\in C_i$ is equivalent to $\lambda_v(x)\le \lambda_w(x)$ for any $v$ in the ancestors of $\ell_i$ and $w$ a sibling of $v$. 

Now we take small enough $\varepsilon>0$ and for any $C_i$ and its corresponding $\ell_i$ consider the full chain from the root $v_0<v_1<\dots < v_m=\ell_i$ and the corresponding affine function:
$$
\lambda_{i,\varepsilon} = \lambda_{v_0} + \varepsilon \lambda_{v_1} + \dots + \varepsilon^m \lambda_{v_m}.
$$
Now it is obvious that the affine partition of $B$ corresponding to $\{\lambda_{i,\varepsilon}\}_{i=1}^k$ tends to $\{C_i\}_{i=1}^k$ when $\varepsilon$ tends to $+0$.
\end{proof}

Even without this lemma it is obvious that Corollary~\ref{kadets-aff} holds for hierarchical affine partitions by induction. Note that a binary partition by hyperplanes is a particular case of a hierarchical affine partition.

\section{The two-dimensional case}
\label{two-dim-sec}

Now we are ready to prove an analogue of the Kadets theorem in the plane. The key property of an inductive partition in the proof of Theorem~\ref{kadets-ind} is actually the following: we consider convex partitions $C_1\cup\dots\cup C_k=B$ that can be extended to a convex partition $V_1\cup\dots\cup V_k=\mathbb R^d$. Then we can translate $V_i$'s with $y$ so that one of the sets $C_i=V_i\cap B$ disappears, remove $C_i$, extend the partition $\{C_j\}_{j\neq i}$ again to a new partition of the whole space, and so on.

In the plane the extension is always possible by the following:

\begin{lemma}
\label{part-ext}
Any convex partition $C_1\cup\dots\cup C_k=B\subset \mathbb R^2$ can be extended to a partition $V_1\cup\dots\cup V_k=\mathbb R^2$.
\end{lemma}

\begin{proof}

The boundary $\partial B$ consists of parts of the boundaries $\partial C_i$. Denote the vertices of this partition by $a_1$, $a_2$, \dots, $a_n$. Denote the polygon $a_1a_2\dots a_n$ by $A$. Note that for some $C_i$'s we may have more than one corresponding part of $\partial B$.

Obviously, from each point $a_i$ it is possible to draw a ray $\ell_i$ outside $B$ with the following property: For any set $C_j$ and its corresponding boundary segment $[a_ia_{i+1}]$ (the indices are understood cyclically $a_{n+1}=a_1$) the union of $C_j$ and the area that is bounded by $[a_ia_{i+1}]$, $\ell_i$, and $\ell_{i+1}$ is convex. For the rays $\ell_i$ one can take the extension of the interior with respect to $B$ side of $C_j$ after $a_i$.

Our goal is to erase parts of the rays $\ell_i$ and obtain a partition of  $\mathbb{R}^2\setminus A$ into $n$ convex parts. At the start the rays may partition $\mathbb{R}^2\setminus A$ into a larger number of parts.

We perform erasing as follows. Suppose $b_j$ is a point of transversal intersection of two rays $\ell_s$ and $\ell_t$ that is closer to $A$ than other points of transversal intersection of the remaining rays. Note that the segments $b_ja_s$ and $b_ja_t$ do not intersect with other rays $\ell_i$ transversally (in this case the point of intersection would be closer to $A$ than $b_j$). Erase part of one of the rays $\ell_s$ or $\ell_t$ after $b_i$, and start new iteration of this process again. Some rays actually become segments, but it does not matter.

After each step we have a convex partition of $\mathbb{R}^2\setminus A$ and finally we obtain a partition into exactly $n$ parts.

After taking union of these parts with their corresponding sets $C_i$ we obtain the required extension of the partition to the whole plane.

\begin{center}
\includegraphics{figure-1.mps}	
\f \label{fig:separation}
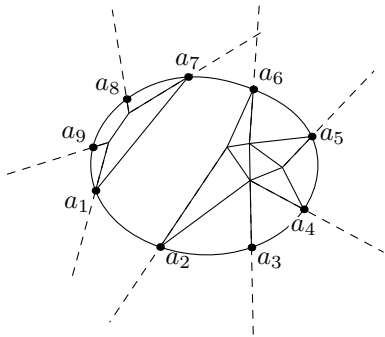
{Extending the partition.}
\end{center}
\end{proof}

Now we are ready to state the result:

\begin{theorem}
\label{kadets-2d}
Let $B\subset \mathbb R^2$ be a convex body and let $C_1\cup\dots\cup C_k = B$ be its convex partition. Then
$$
\sum_{i=1}^k r_B(C_i) \ge 1.
$$
\end{theorem}

\begin{proof}
We extend the partition $C_1\cup\dots\cup C_k = B$ to $V_1\cup\dots\cup V_k=\mathbb R^2$ by Lemma~\ref{part-ext}. Then the function 
$$
r(y) = \sum_{i=1}^k r_B(B\cap (V_i+y))
$$
is again concave, so by varying $y$ we can make one of $B\cap (V_i+y)$ have empty interior without increasing $r(y)$. Then we omit $V_i$, obtain a partition of $B$ into fewer parts, and use the inductive assumption.
\end{proof}

\section{Possible extension to coverings}

Theorem~\ref{kadets-2d} is quite close to the plane case of the Bang conjecture, which we restate here: If $B\subset \mathbb R^2$ is covered by a set of planks $W_1\cup\dots\cup W_k\supseteq B$ then $\sum_{i=1}^k r_B(W_i) \ge 1$. The key difference is that in the Bang conjecture we have a covering, not a partition. Intuitively, partition is something smaller than covering and therefore has smaller sum of ``inradii''. But already in the case of $\mathbb R^2$ there exist coverings that do not contain partitions. A simple example is a set of planks $C_i$ passing thorough the center of a disk $B$, forming a ``sunflower'' so that each of the sets $C_i\cap\partial B$ consists of two disjoint arcs and these arcs partition $\partial B$.

It is easily verified that Theorem~\ref{kadets-ind} holds (with the same proof literally) for coverings instead of partitions if we define an inductive covering by:

\begin{definition}
Call a convex covering (by closed sets) $V_1\cup\dots\cup V_k$ of $\mathbb R^d$ \emph{inductive} if for any $1\le i\le k$ there exists an inductive covering $W_1\cup\dots\cup W_{i-1}\cup W_{i+1}\cup\dots\cup W_k$ such that $W_j\subseteq V_j\cup V_i$ for any $j\neq i$. A covering by one set $V_1=\mathbb R^d$ is assumed to be inductive.
\end{definition}

Returning to the Bang conjecture we see the unpleasant thing: If we cover some part of $\mathbb R^2$ by planks and the remaining part is covered by the corresponding (possibly infinite) polygons, then none of the polygons can be deleted, so this is an example of a non-inductive covering of the plane.

\section{Notes on the spherical Kadets theorem}

In~\cite{bezdek2010covering} K\'{a}roly~Bezdek and Rolf~Schneider proved the following version of the Kadets theorem in the spherical geometry:

\begin{theorem}[K.~Bezdek, R.~Schneider, 2010]
\label{bez-schn}
If a the sphere $\mathbb S^n$ is covered by spherical convex sets $K_i$ then we have the inequality for the inradii:
$$
\sum_i r(K_i) \ge \pi.
$$
\end{theorem}

This theorem gives rise to the following:

\begin{problem}[K.~Bezdek, R.~Schneider, 2010]
\label{kadets-on-sphere}
Suppose $B_\rho\subset \mathbb S^n$ is a ball of radius $\rho$ in the spherical geometry. Suppose $B_\rho$ is covered by spherical convex sets $K_i$; prove that
$$
\sum_i r(K_i)\ge \rho.
$$ 
\end{problem}

As it is noted in~\cite{bezdek2010covering} Theorem~\ref{bez-schn} solves this problem for $\rho\ge \pi/2$, the solution being essentially volumetric. But if $\rho\to 0$ then this problem approaches the original Kadets theorem, which has no volumetric solution for $n>2$. So it seems that solving this problem for $\rho$ in the range $(0, \pi/2)$ must require a new approach.

Let us outline the proof of Theorem~\ref{bez-schn}. This proof is essentially the same as the proof given in~\cite{bezdek2010covering}; but we simplify it and split into several lemmas, some of which may be of interest on their own.

\begin{lemma}
\label{star-corr}
Let $\mu$ be a spherically symmetric absolute continuous measure on $\mathbb R^n$, $B$ be a ball centered at the origin, and $T$ be a $0$-starshaped body. Then
$$
\mu(B\cap T) \mu(\mathbb R^n) \ge \mu(B) \mu(T).
$$
\end{lemma}

\begin{proof}
The proof will use a very simple case of the needle decomposition (see~\cite{nazarov2002geometric} for example). Let us split $\mathbb R^n$ into convex cones $V_i$ of equal measures $\mu(V_i)$. Note that the sets $V_i\cap B$ will also have equal measures because of the spherical symmetry of $\mu$. The lemma will follow from the inequality:
\begin{equation}
\label{cone-part-ineq}
\mu(B\cap T\cap V_i) \mu(V_i) \ge \mu(B\cap V_i) \mu(T\cap V_i)
\end{equation}
by summation. The partition can be made so that every $V_i$ gets arbitrarily close to a $1$-dimensional ray, and the limit case of (\ref{cone-part-ineq}) becomes an inequality for nonnegative functions:
$$
\int_0^{\min\{x, y\}} f(t)\, dt \cdot \int_0^{+\infty} f(t)\, dt \ge \int_0^x f(t)\, dt \cdot \int_0^y f(t)\, dt,
$$
which simply follows from the observation that $\min\{X, Y\} Z \ge XY$ for any $X,Y\in [0, Z]$.
\end{proof}

\begin{lemma}
\label{sph-corr}
Let us work in the spherical geometry. Suppose $H\subset \mathbb S^n$ is a hemisphere with center $o$, $B$ is a ball of radius $\le \pi/2$ centered at $o$, $T$ is an $o$-starshaped body in $H$. Then
$$
\sigma(B\cap T) \sigma(H) \ge \sigma(B) \sigma(T)
$$
for the standard measure $\sigma$ on the sphere.
\end{lemma}

\begin{proof}
Follows from Lemma~\ref{star-corr} by central projection of $H$ onto $\mathbb R^n$ such that $o$ goes to $0$.
\end{proof}

\begin{lemma}
\label{eps-neighborhoods}
Let $X$ be s subset of the sphere $\mathbb S^n$ not contained in an open hemisphere, and $X_0$ be a set consisting of two antipodal points on the sphere. Then for their $\varepsilon$-neighborhoods (in the spherical geometry) we have:
$$
\sigma(X + \varepsilon) \ge \sigma(X_0 + \varepsilon).
$$
\end{lemma}

\begin{proof}
Without loss of generality let $X=\{o_1, \ldots, o_m\}$ be finite. Consider the hemispheres $H_i$ with respective centers $o_i$ and the Voronoi regions $V_i$ of $o_i$. Note that $V_i\subseteq H_i$ for every $i$. Denote the measure of the whole sphere $\mathbb S^n$ by $\sigma_n$.

Then by Lemma~\ref{sph-corr}
$$
\sigma(V_i\cap B_{o_i}(\varepsilon)) \frac{\sigma_n}{2} \ge \sigma(B_{o_i}(\varepsilon)) \sigma(V_i),
$$
hence 
$$
\sigma(V_i\cap (X+\varepsilon)) \frac{\sigma_n}{2} \ge \sigma(B_{o_i}(\varepsilon)) \sigma(V_i),
$$
and then by summing over $i$ and multiplying by $2$:
$$
\sigma(X+\varepsilon) \sigma_n \ge 2 \sigma(B_*(\varepsilon)) \sigma_n,
$$
where $B_*(\varepsilon)$ is any ball (on the sphere) of radius $\varepsilon$. So we obtain:
$$
\sigma(X+\varepsilon) \ge 2 \sigma(B_*(\varepsilon)) = \sigma(X_0+\varepsilon).
$$
\end{proof}

\begin{lemma}
\label{isoperimetry-eu}
Let $\mu$ be an absolute continuous spherically symmetric measure on $\mathbb R^n$.
Suppose $K$ is a convex body in $\mathbb R^n$ with inscribed ball $B$, centered at the origin. Then 
$$
\mu(K) \le \mu (T)
$$
where $T$ is a plank with inscribed ball $B$.
\end{lemma}

\begin{proof}
Representing the measure $\mu$ as an integral it is enough to prove the inequality:
\begin{equation}
\label{sph-ineq}
\sigma(K\cap S) \le \sigma (T\cap S)
\end{equation}
for any sphere $S$ centered at the origin. Let the radii of $B$ and $S$ be $r$ and $R$ respectively. For $R\le r$ the inequality (\ref{sph-ineq}) is obvious, so we consider $R>r$. 

The set $T\cap S$ is the complement of the $\varepsilon$-neighborhood of two opposite points $X_0$ in $S$, where $\varepsilon = \arccos{r/R}$. Since $K$ has $B$ as the inscribed ball, the set $X'=\partial K\cap B$ contains the origin in its convex hull. It is easy to see that the set $X = R/r X'$ is not contained in a hemisphere, and its $\varepsilon$-neighborhood is disjoint with $K\cap S$. By Lemma~\ref{eps-neighborhoods}:
$$
\sigma(X+\varepsilon) \ge \sigma (X_0+\varepsilon),
$$
hence 
$$
\sigma(K\cap S) \le \sigma (S\setminus(X+\varepsilon)) \le \sigma (S\setminus (X_0+\varepsilon)) = \sigma(T\cap S),
$$
which is exactly (\ref{sph-ineq}).
\end{proof}

Now we deduce the following (the same as~\cite[Theorem~2]{bezdek2010covering}):

\begin{lemma}
\label{isoperimetry}
Suppose $K$ is a convex body in $\mathbb S^n$ with inscribed ball $B$. Then 
$$
\sigma(K) \le \sigma (K_0),
$$
where $K_0=H_0\cap H_1$ is an intersection of two hemispheres with inscribed ball $B$. Note that $\sigma(K_0)$ equals $\frac{\alpha \sigma_n}{2\pi}$, where $\alpha$ is the angle between $H_0$ and $H_1$.
\end{lemma}

\begin{proof}
Obtained from Lemma~\ref{isoperimetry-eu} by central projection that takes the center of $B$ to the origin in $\mathbb R^n$.
\end{proof}

Now Theorem~\ref{bez-schn} follows from Lemma~\ref{isoperimetry} by bounding from above the volume of every $K_i$ in terms of $r(K_i)$.

\section{The hyperbolic Kadets theorem}

It is interesting that the Kadets theorem does not hold for hyperbolic space unlike the spherical case mentioned above. We skip the calculation here because of the negativity of this result, but the figures below should be sufficiently convincing.

Consider a sufficiently large disk $\Omega$ and a regular hexagon inscribed in it. Let us cover this disk by two convex shapes, which are drawn in Fig.~\ref{fig:hypercovering} (this is the Poincar\'e model). The maximal inscribed disk of a shape is drawn by a dashed line. Since it does not contain the center of $\Omega$ its radius is less than half of the radius of $\Omega$. 

\begin{center}
\includegraphics{figure-2.mps}
\hskip 1cm
\includegraphics{figure-3.mps} 
\f{Disk covering.} \label{fig:hypercovering}
\end{center}

Note that this counterexample uses essentially that $C_1$ and $C_2$ do intersect. The authors do not know whether the Kadets theorem holds for partitions in the hyperbolic space.

\section{Conjectures}

Let us mention three conjectures, which are related to the theme of the article. The first two conjectures belong to Mikhail~Smurov, who stated them in Kvant~\cite{smurov1998pokritie1_en}, a Russian journal for high school students.

\begin{conjecture}[M.~Smurov, 1998]
\label{con:smurov ball}
There exists a constant $C_d$ possibly depending on the dimension $d$ with the following property: For any collection of planks in $\mathbb R^d$ with sum of widths at least $C_d$, it is possible to translate them so that they cover the unit ball.
\end{conjecture}

In~\cite{smurov1998pokritie1_en} it is proved that $C_2<2+\pi$ and the case $d>2$ remains open. Here is another conjecture:

\begin{conjecture}[M.~Smurov, 1998]
\label{con:smurov plane}
Suppose planks with sum of widths $1$ in the plane are given. Than for any convex body $B$ with perimeter not greater than $2$ it is possible to translate the planks in so that they cover $B$.
\end{conjecture}

The following conjecture would be a strengthening of the Bang theorem for planks passing through the center of a ball. In this case we may pass to the sphere bounding the ball, introducing the following notation. Suppose $S'$ is an equatorial codimension $1$ subsphere of the sphere $\mathbb S^d$ and $S'_\varepsilon$ is its $\varepsilon$-neighborhood (in the spherical geometry). Call $S'_\varepsilon$ a \emph{plank on the sphere} and call $2\varepsilon$ its \emph{width}.

\begin{conjecture}
\label{con:karasev}
Suppose $\mathbb S^d$ is covered by planks. Then the sum of widths of these planks is at least $\pi$.
\end{conjecture}

\begin{remark}
Note that this result does not follow from Theorem~\ref{bez-schn} because the spherical planks are not convex in spherical geometry.
\end{remark}


\end{document}